\documentclass[12pt]{amsart}
\usepackage{amsmath, amssymb, mathtools}
\usepackage{geometry}
\usepackage[utf8]{inputenc}
\usepackage[T1]{fontenc}
\usepackage{color}
\usepackage{graphicx}
\usepackage[colorlinks=true,citecolor=red,linkcolor=black,urlcolor=blue]{hyperref}

 \makeatletter
 \newtheorem{theorem}{Theorem}[section]

\newtheorem{lemma}[theorem]{Lemma}

\newtheorem{cor}[theorem]{Corollary}
\newtheorem{definition}[theorem]{Definition}

\newtheorem{remark}[theorem]{Remark}

\newtheorem{conjecture}[theorem]{Conjecture}
\newtheorem{note}[theorem]{Note}

 \numberwithin{equation}{section}

\begin{document}
\title{Vinogradov's Conjecture and Beyond}
\author{Shivarajkumar}
\address{Department of Mathematics\\
The National Institute of Engineering\\
Mysuru, India.}
\email{shivarajmacs@gmail.com}
\begin{abstract}
In this paper, if prime $p\equiv3\pmod 4$ is sufficiently large then we prove an upper bound on the number of occurences of any arbitrary pattern of quadratic 
residues and nonresidues of length $k$ as $k$ tends to $\lceil \log_2 p\rceil$.
As an immediate consequence, it proves that, there exist a constant $c$ 
such that, the least nonresidue for such primes is at most $c\lceil \log_2 p\rceil$.
\end{abstract}
\maketitle
\section {Introduction}
An integer $r$ is called a quadratic residue modulo prime $p$, 
if there exists an integer $x$ such that
$$x^2\equiv r\pmod p.$$
Otherwise, $r$ is called a nonresidue modulo $p$. 
\par There are two interesting problems which dominate the theory of distribution of quadratic residues and nonresidues.
\begin{enumerate}
 \item If $R$ (resp. $N$) is the maximum length of consecutive quadratic residues
 (resp. nonresidues) modulo $p$, then what is $R$(resp. $N$)?
\item 
For each prime $p$, let $n(p)$ denote the least natural number that is not a quadratic 
residue modulo $p$. Then what is $n(p)$?
\end{enumerate}
Much of the effort has been devoted to answer 
the second probelm. Gauss proved the first nontrivial bound
on this problem: he proved that if $p\equiv 1 \pmod 8$, then
the least nonresidue is less than $2\sqrt{p}+1$.

Vinogradov \cite{vinogr} established the asymptotic bound  $$n(p)\ll p^{\frac{1}{\sqrt{2e}}}\log^2 p$$
for all primes and made  the following conjecture.
\begin{conjecture}
For any fixed $\epsilon >0$, we have $n(p)\ll p^{\epsilon}$. 
\end{conjecture}
In 1942, Linnik \cite{linnik} proved that, this conjecture follows from the 
Generalised Riemann Hypothesis (GRH). In 1952, Ankeny \cite{ankeny} improved the bound further to $n(p)\ll \log^2 p$. However the conjecture remains open unconditionally.

In 1957, Burgess \cite{berg} proved that $n(p)\ll p^{\frac{1}{4\sqrt{e}}}$ without assuming GRH. 
Till today this remains to be the best known bound.
\par On the other hand, in 1990, Graham and Ringrose \cite{graham} proved that, for infinitely many primes $n(p) \geq c \log p$ \ \ .

\par Recently, Tao \cite{tao}, connected one of the standard conjectures in sieve theory called the Elliot-Halberstam conjecture by proving Elliot-Halberstam conjecture implies Vinogradov's conjecture. For more details on Conjecture 1.1, refer to Tao \cite{tao1}.
\par In 1992, by exploring some combinatorial implications of Weil's bound on charecter sums, Peralta \cite{peralta} proved the following.
\begin{theorem}
The number of occurences of arbitrary patterns of quadratic residues and
nonresidues of length $k$ in $\mathbb{Z}_p$ is in the range 
$\frac{p}{2^k}\pm k(3+\sqrt{p})$	
\end{theorem}	
In this paper, by using purely combinatorial arguments, we discuss about the number of occurences of patterns of length $k$ for primes  
$p\equiv 3\pmod 4$. As an immediate consequence, we prove that, there exists a constant $c$ such that $n(p)\leq c \lceil \log_2 p\rceil$.
This proves beyond Vinogradov's conjecture.

In the next section, we discuss about the number of occurences of all possible 
patterns of quadratic residues and nonresidues of
length 2 and 3 for primes $p\equiv 3\pmod 4$.

\section{Preliminaries}
For a prime $p$, $\mathbb{Z}_p$
denote the group of integers modulo prime $p$.
\par Let $a, a+d, a+2d,...,a+(k-1)d$ be a $k-$arithmetic progression with the common difference $d$, where $a+id \in \mathbb{Z}_p\backslash\{0\}$ for all $0\leq i \leq k-1$.
We say that, the above arithmetic progression is a $k-$arithmetic progression of quadratic residues if for every $i$, $a+id$ is a quadratic residue. Similarly, if $a+id$ is a nonresidue for every $i$ then we say that it is $k-$arithmetic progression of nonresidues. We also define, the following.
\par Let $a, a+1, a+2,..., a+(k-1)$ be $k-$consecutive nonzero elements of $\mathbb{Z}_p$. Then we define 

$ Pattern \ of \ length \ k = \begin{cases}
a+i  & \text{is quadratic residue for all $i$ or }\\
a+i & \text{is nonresidue for all $i$ or }\\
a+i & \text{is quadratic residue for some $i$ and nonresidue for other $i$}
\end{cases}$\\

Also, we define, two arithmetic progressions $a, a+d, a+2d,..., a+(k-1)d$ and $b, b+d, b+2d,..., b+(k-1)d$ tobe the same if $a+id$ and  $b+id$ are simultaneously quadratic residues or nonresidues for $0\leq i \leq k-1$.

\par In this section, for a given prime $p\equiv 3\pmod 4$, we find the exact 
number of occurences of all possible patterns 
of length 2 and 3 in $\mathbb{Z}_p$.
\par Let $r$ and $n$ be the quadratic residue and the nonresidue respectively.
Following are the set of all possible patterns of length 2 and 3.
\begin{center}
\begin{tabular}{|l|c|r|r|r|p{1.8cm}|}
  \hline
 \it{rr} & \it{nr} \\
 \it{rn} & \it{nn}   \\
 \hline
\end{tabular}
\begin{tabular}{|l|c|r|r|r|p{1.8cm}|}
  \hline
 \it{rrr} & \it{rrn} & \it{nrr} & \it{nrn} \\
 \it{rnr} & \it{rnn} & \it{nnr} & \it{nnn} \\
 \hline
\end{tabular}
\end{center}
\begin{definition}\cite{lint}
 A $k-$element subset $D$ of a finite group $(G, *)$ of order $v$ is called a 
$(v,k,\lambda)$ difference set in $G$ provided that the multiset 
$\{d_1*d_2^{-1} : d_1,d_2\in D, d_1 \neq d_2 \}$ contains 
each nonidentity elelment of $G$  exactly $\lambda$ times.

A difference set $D$ in an additive group $G$ is called a skew Hadamard difference 
set if and only if $G$ is the disjoint union of $D$, $-D$, $\{0\}$.
\end{definition}
\begin{note}
If $p\equiv 3\pmod 4$, then the set of all quadratic residues of $\mathbb{Z}_p$ forms
a skew Hadamard difference set (SHDS) with $\lambda = \frac{p-3}{4}$. Therefore, 
there are exactly $\frac{p-3}{4}$ number of $2-$arithmetic progressions of quadratic residues
with common difference $d$ for every $d\in \mathbb{Z}_p\backslash\{0\}$.
\end{note}
From the above note, we can see that, the number of 2-arithmetic progressions of quadratic residues is exactly $\frac{p-3}{4}$.
\par In 2016 \cite{shiv}, we have proved the following.
\begin{theorem}\cite{shiv}
 For $p\equiv 3\pmod 4$ the number of
$3-$arithmetic progressions of quadratic residues with 
common difference $d$ is exactly 
$\lfloor\frac{p-3}{8}\rfloor$ for every 
$d\in \mathbb{Z}_p\backslash\{0\}$.
\end{theorem}

\begin{note}
	We know that, for any given prime $p\equiv 3(mod\ 4)$, if $x$ is a quadratic residue
	then $-x$ is nonresidue. 
\end{note}
Therefore, the number of $2-$arithmetic progressions, $3-$arithmetic progressions of nonresidues with common difference $d$ for every $d\in \mathbb{Z}_p\backslash\{0\}$ are 
$\frac{p-3}{4}$, $\lfloor\frac{p-3}{8}\rfloor$ respectively. \\
\par From above Theorem 2.3, Note 2.2 and Note 2.4, we have the following.
\begin{cor}
For $p\equiv 3\pmod 4$ the number of occurences of any pattern of length $3$ is either $\lfloor\frac{p-3}{8}\rfloor$ or $\lceil\frac{p-3}{8}\rceil$.
\end{cor}
\begin{proof}
Since the number of occurences of patterns of the type $rrr$ is $\lfloor\frac{p-3}{8}\rfloor$, from Note 2.2, we can see that, the number of occurences of patterns of the type $rnr$ and $rrn$ are $\lceil\frac{p-3}{8}\rceil$. Therefore, the number of occurences of patterns of type $rnn$ is $\lfloor\frac{p-3}{8}\rfloor$. From Note 2.4,
the number of occurences of patterns of type $nnn$, $nrr$ are $\lfloor\frac{p-3}{8}\rfloor$ and the number of occurences of patterns of type $nnr$, $nrn$ are $\lceil\frac{p-3}{8}\rceil$.	
\end{proof}
Similarly, one can prove the following.
\begin{cor}
For $p\equiv 3\pmod 4$ the number of occurences of any pattern of length 2 is either $\frac{p-3}{4}$ or $\frac{p-3}{4}+1$.
\end{cor}
In the next section, we discuss about the number of occurences
of arbitrary patterns of length $k>3$.

\section{Main results}
In this section, we prove that, if $p$ is large enough then there exists a constant $c$ such that the number 
of occurences of any pattern of length $k$ is at most c as $k$ tends to 
$\lceil \log_2 p\rceil$.

We know that there are $2^k$ number of patterns of length $k$. Let us classify
all the possible patterns of length $k$ into four types and each of 
size $2^{k-2}$ as follows.
\begin{enumerate}
 \item Patterns starts with quadratic residue and ends with quadratic residue.
 \item Patterns starts with quadratic residue and ends with nonresidue.
 \item Patterns starts with nonresidue and ends with quadratic residue.
 \item Patterns starts with nonresidue and ends with nonresidue.
\end{enumerate}

From Note 2.4, the number of occurences
of all patterns of type 3 and 4 is immediate if we know the number of occurences
of all patterns of type 1 and 2.

Now, we discuss about the number of occurences of patterns
of type 1 and same argument is valid for type 2 as well.

Let $(\frac{p}{2^k})=\lfloor \frac{p}{2^k} \rfloor$ or 
$\lceil \frac{p}{2^k} \rceil$. Consider all possible patterns of type 1 with length $k$. Since we do not know the 
exact number of occurences of patterns of length $k\geq 4$, let us assume that, 
the number of occurences of a pattern 
be $(\frac{p}{2^k})+e(i,k)$,  where, 
$1 \leq i \leq 2^{k-2}$ and $e(i,k) \in \mathbb{Z}$
for every $i$ and $k$.

Without loss of generality, let
\begin{equation}\label{11}
\underbrace{(\frac{p}{2^k})+(\frac{p}{2^k})+...+(\frac{p}{2^k})}_{\text{$2^{k-2}$ terms}}=
\frac{p-3}{4}.
\end{equation}
From (\ref{11}) and Note 2.2, we can see that, 
\begin{equation}\label{12}
\sum\limits_{i=1}^{2^{k-2}}e(i,k) = 0.
\end{equation}
Now we prove the following.
\begin{lemma}
	If $e(i,k)$ is zero for every $i$, $k$
	then $(\frac{p}{2^k})$ is at most 1
	as $k$ tends to $\lceil \log_2p\rceil$.
\end{lemma}
\begin{proof}
	Let the number of occurences of a pattern 
	be $(\frac{p}{2^k})+e(i,k)$,  where, 
	$1 \leq i \leq 2^{k-2}$ and $e(i,k) \in \mathbb{Z}$
	for every $i$ and $k$. If $e(i,k)$ is zero for every $i,\ k$
	then  $(\frac{p}{2^k})\leq 1$ as $k$ tends to $\lceil \log_2p\rceil$.	
\end{proof}	
\begin{lemma}
If $e(i,k)$ is nonzero then it must be goverened by a rule.
\end{lemma}
\begin{proof}

Let $k$ be $(p_1p_2...p_s)n+1$, 
where, $p_1,p_2,...,p_s$ 
are distinct primes such that $n$ and $s$ are positive integers. Then
for any arbitrary pattern (type 1) of length $k$, there are arithmetic
progressions of lengths $(p_2p_3...p_s)n+1$, $(p_1p_3...p_s)n+1$,..., $(p_1p_2...p_{s-1})n+1$
with common differences $p_1$, $p_2$,..., $p_s$ respectively.

Corresponding to each $j$, \ $1\leq j \leq s$, let us divide $2^{k-2}$ patterns into $2^{(p_1...p_{j-1}p_{j+1}...p_s)n-1}$ groups. In each group there are exactly $2^{k-[(p_1...p_{j-1}p_{j+1}...p_s)n+1]}$ patterns such that arithmetic progressions with common difference $p_j$
are same.
\par Let $j$ be fixed. For every $1\leq i \leq 2^{k-2}$, there exists 
$1\leq m \leq 2^{(p_1...p_{j-1}p_{j+1}...p_s)n-1}$ and $1\leq n \leq 2^{k-[(p_1...p_{j-1}p_{j+1}...p_s)n+1]}$ such that $e(i, k)=e(m_n, k)_{j}$. Therefore, 
$$\sum\limits_{n=1} ^{2^{k-[(p_1...p_{j-1}p_{j+1}...p_s)n+1]}}[(\frac{p}{2^k})+ 
e(m_n,k)_{j}]=\\ (\frac{p}{2^{(p_1...p_{j-1}p_{j+1}...p_s)n+1}})
+ e(m, [(p_1...p_{j-1}p_{j+1}...p_s)n+1]).$$
Let
\begin{equation}\label{1}
\underbrace{[(\frac{p}{2^k})+r_1]+[(\frac{p}{2^k})+r_2]+...
	+[(\frac{p}{2^k})+r_n]}_{\text{$2^{k[(p_1...p_{j-1}p_{j+1}...p_s)n+1]}$ terms}}=
(\frac{p}{2^{(p_1...p_{j-1}p_{j+1}...p_s)n+1}}),
\end{equation}
where, $r_i$ can choose either 0, 1 or -1 and $r_i$ need not be equal to $r_j$ for every $1 \leq i \leq j \leq 2^{k-[(p_1...p_{j-1}p_{j+1}...p_s)n+1]}$.
Hence, for a fixed $j, \ m$, 
\begin{equation}\label{2}
\sum\limits_{n=1}^{2^{k-[(p_1...p_{j-1}p_{j+1}...p_s)n+1]}} 
[e(m_n, k)_{j}+r_n]=
e(m, ((p_1...p_{j-1}p_{j+1}...p_s)n+1)).
\end{equation}
One can also see that, patterns shuffles across different groups for diffferent $j$. Therefore, if $s$ is large enough then $e(i, k)$ can't be choosen randomly. Hence, if $e(i, k)$ is nonzero then for every $i$, $k$ it must be defined by a rule.
\end{proof}	

\begin{remark}
We know that, $(\frac{p}{2^k})$ is at most 1 as $k$ tends to $\lceil \log_2p\rceil$. Hence for every $1\leq i \leq 2^{k-2}$, $e(i, k)$ must be greater than $-1$.
\end{remark}
Now, we prove the following.
\begin{lemma}
As $k$ tends to $\lceil \log_2p\rceil$ for some $i$ if $e(i,k)$ is any arbitrary positive function of $p$ then there exists $l \neq i$ such that $e(l, k)$ is a negative function of $p$. 
\end{lemma}
\begin{proof}
As $k$ tends to $\lceil \log_2p\rceil$, suppose that, $e(i,k)$ be any arbitrary positive function of $p$, say $f(p)$. Then from (3.2) and Remark 3.3 there exist $j_1, j_2,..., j_{f(p)}$ such that $e(j_r, k)=-1$
for every $1 \leq r \leq f(p)$.
\par But according to Lemma 3.2, $e(i, k)$ can't be choosen randomly. It must be governed by a rule and there exists no rule which satisfy the above condition. Therefore, there exists $l \neq i$ such that $e(l, k)$ must be a negative function of $p$.
\end{proof}	
From  Lemma 3.1, Lemma 3.2 and Lemma 3.4 we have the following.	
\begin{theorem}
If $p$ is sufficiently large then there exists a constant $c$ such that the number of occurences 
of patterns of length $k$ as $k$ tends to $\lceil \log_2p\rceil$
at most c.
\end{theorem}

For $p\equiv 3\pmod 4$, from the above theorem, 
we have the following.
\begin{cor}
If $p$ is sufficiently large then $n(p) \leq c\lceil \log_2 p\rceil$.
\end{cor}
The above corollary proves beyond Vinogradov's conjecture.

\end{document}